\documentclass[11pt]{amsart}
\usepackage{graphicx,subcaption}
\usepackage{amsmath} 
\usepackage{amsthm,amsfonts,amssymb,mathrsfs,amscd,amstext,amsbsy} 
\usepackage{epic,eepic} 
\usepackage{yfonts}
\usepackage{paralist,enumerate}
\usepackage[all]{xy}
\usepackage{hyperref}
\hypersetup{colorlinks}

\usepackage{tikz}
\usetikzlibrary {positioning}

\usetikzlibrary{arrows}
\tikzset{    vertex/.style={circle,draw,minimum size=1.5em},    edge/.style={->,> = latex'}}

\newcommand{\D}{\displaystyle}

\newtheorem{theorem}{Theorem}[section]

\newtheorem{lemma}[theorem]{Lemma}

\newtheorem*{Definition*}{Definition}
\numberwithin{equation}{section}

\def\qed{\hfill \ifhmode\unskip\nobreak\fi\quad\ifmmode\Box\else$\Box$\fi\\ }

\begin{document}

\title[10-dimensional almost complex $S^1$-manifold with fixed points]{Lower bound on the number of fixed points for circle actions on 10-dimensional almost complex manifolds}
\author{Donghoon Jang}
\thanks{2020 Mathematics Subject Classifications: 58C30, 32Q60, 37C25}
\thanks{keywords: almost complex manifold, circle action, fixed point, Chern number}
\thanks{This work was supported by a 2-Year Research Grant of Pusan National University.}
\address{Department of Mathematics and Institute of Mathematical Science, Pusan National University, Pusan, Korea}
\email{donghoonjang@pusan.ac.kr}

\begin{abstract}
For a circle action on a compact almost complex manifold with a fixed point, the lower bound on the number of fixed points is known in dimension up to 12 except 10.
In this paper, we show that if the circle group acts on a 10-dimensional compact almost complex manifold with a fixed point, then there are at least 6 fixed points. This minimum is attained by $\mathbb{CP}^5$ and $S^6 \times \mathbb{CP}^2$. We establish this lower bound by showing that there does not exist a circle action on a 10-dimensional compact almost complex manifold with 4 fixed points.
\end{abstract}

\maketitle

\section{Introduction}

An \textbf{almost complex manifold} is a manifold $M$ together with a smooth bundle map $J$ (called the \textbf{almost complex structure}) on the tangent bundle $TM$ of $M$ such that the restriction $J_m$ of $J$ on each tangent space $T_mM$ has a linear complex structure, that is, $J_m^2=-\textrm{Id}_{T_mM}$. By definition, every almost complex manifold is even dimensional. We say that an action of a group $G$ on an almost complex manifold $(M,J)$ \textbf{preserves} the almost complex structure $J$ if $dg \circ J=J \circ dg$ for all $g \in G$. Throughout this paper, we assume that any group action on an almost complex manifold preserves the almost complex structure.

Let the circle group act on a $2n$-dimensional compact, connected almost complex manifold $M$. Let $k$ be the number of fixed points. When $k$ is small, possible dimensions for $M$ and known examples of such a manifold $M$ are as in Table \ref{table1}; also see Theorem \ref{few}. Circle actions with few fixed points on different types of manifolds have been studied; for those results, we refer to \cite{A, DW, GS, J1, J3, J4, J6, J8, KL, K1, K2, LM, M, PT, T, W}.

\begin{table}[]
\begin{tabular}{|l|l|l|}
\hline
Number of fixed points & Possible $\dim M$ & Example \\ \hline
1  & 0 & point \\ \hline
2  & 2 and 6   & $S^2$ and $S^6$ \\ \hline
3  & 4 & $\mathbb{CP}^2$ \\ \hline
\end{tabular}
\caption{Possible dimensions for small number of fixed points} \label{table1}
\end{table}

It is a natural question to ask what the minimal number of fixed points is for such a manifold $M$ in a fixed dimension. When the dimension of the manifold is small, the minimum for the number of fixed points and a known example that attains this minimum are as in Table \ref{table2}. Note that in \cite{GKZ}, Goertsches, Konstantis, and Zoller constructed an 8-dimensional manifold with 4 fixed points, which is not $S^2 \times S^6$; it is a non-trivial $S^2$-bundle over $S^6$. Godinho, Pelayo, and Sabatini studied this question for the class of manifolds whose Chern number $\int_M c_1 c_{n-1}$ vanish \cite{GPS}.

This question is related to a conjecture of Kosniowski which asserts that if the circle group acts on a $2n$-dimensional compact unitary manifold $M$  with $k$ fixed points and $M$ does not bound equivariantly, then $\lfloor \frac{n}{2} \rfloor +1 \leq k$ \cite{K2}. Kosniowski \cite{K2} and Musin \cite{M} showed that if $k=2$ then $n=1$ or $3$. Wiemeler showed that if $k=3$ then $n=2$ \cite{W}. Thus, the Kosniowki conjecture holds in dimension up to 14.

\begin{table}[]
\begin{tabular}{|l|l|l|}
\hline
Dimension of $M$ & Minimum for $|M^{S^1}|$ & Example \\ \hline
0  & 1 & point \\ \hline
2  & 2 & $S^2$ \\ \hline
4  & 3 & $\mathbb{CP}^2$ \\ \hline
6  & 2 & $S^6$  \\ \hline
8  & 4 & $S^2 \times S^6$ \\ \hline
12 & 4 & $S^6 \times S^6$ \\ \hline
\end{tabular}
\caption{Minimal number of fixed points in low dimensions} \label{table2}
\end{table}

Let $M$ be a compact, connected almost complex manifold endowed with a circle action. If the number $k$ of fixed points is odd, then the dimension of the manifold is a multiple of $4$; see Lemma \ref{even}. If $k=2$, by Theorem~\ref{few} (2), $M$ is the $2$-sphere or $\dim M=6$. Now, suppose that the dimension of $M$ is 10. By the discussions, $k$ is even and $k \neq 2$. In this paper, we establish an exact lower bound for the number of fixed points for a 10-dimensional manifold.

\begin{theorem} \label{bound}
Let the circle group act on a 10-dimensional compact almost complex manifold. If the action has a fixed point, then there are at least 6 fixed points.
\end{theorem}

We prove Theorem~\ref{bound} in the last section.
The examples of 10-dimensional $M$ with 6 fixed points are the standard linear action on $\mathbb{CP}^5$ and a diagonal action on $\mathbb{CP}^2 \times S^6$; note that the Todd genus of $\mathbb{CP}^5$ is 1, while that of $\mathbb{CP}^2 \times S^6$ is 0. 
Circle actions on complex manifolds and symplectic circle actions on symplectic manifolds are particular examples of circle actions on almost complex manifolds. Thus, the conclusion of Theorem \ref{bound} holds for complex manifolds and symplectic manifolds.

Circle actions on compact almost complex manifolds with 4 fixed points are known to exist in dimension up to 12 except 10.
We list examples of actions with 4 fixed points in Table~\ref{table3} in those dimensions.
In dimension 0 and 2, such an example is disconnected. Dimension 6 has more examples than other dimensions, the complex projective space $\mathbb{CP}^3$, a complex quadric, Fano 3-folds $V_5$ and $V_{22}$, $S^6 \sqcup S^6$ (disconnected), blow up of a point in $S^6$ \cite{J4}, $S^2 \times S^4$ \cite{KL}. Dimension 8, as mentioned, has two known examples, $S^2 \times S^6$ and the non-trivial $S^2$-bundle over $S^6$ \cite{GKZ}.

\begin{table}[]
\begin{tabular}{|l|l|l|}
\hline
Dimension  &  Example \\ \hline
0  & 4 points \\ \hline
2  & $S^2 \sqcup S^2$ \\ \hline
4  & Hirzebruch surfaces \\ \hline
6  & $\mathbb{CP}^3$, complex quadric, Fano 3-fold, \cite{J4}, \cite{KL}  \\ \hline
8  & $S^2 \times S^6$, \cite{GKZ} \\ \hline
12 & $S^6 \times S^6$ \\ \hline
\end{tabular}
\caption{Examples of actions with 4 fixed points in low dimensions} \label{table3}
\end{table}

Therefore, 10 is the first dimension for which we do not know whether there exists an action with 4 fixed points. We show that such an action does not exist.

\begin{theorem} \label{main}
There does not exist a circle action on a 10-dimensional compact almost complex manifold with exactly 4 fixed points.
\end{theorem}

The main ingredients for the proof of Theorem \ref{main} are the Chern numbers, the Todd genus, and the ABBV localization theorem (Theorem \ref{t21}). We outline the proof of Theorem \ref{main}. For this, let the circle group act on a $10$-dimensional compact almost complex manifold $M$ with 4 fixed points. 
The Todd genus $\mathrm{Todd}(M)$ of $M$ satisfies $\mathrm{Todd}(M)=\int_M \frac{1}{1440} (-c_1c_4+c_1^2c_3+3c_1c_2^2-c_1^3c_2)$.
By using the previous work of \cite{J5} and the ABBV localization theorem, we prove Lemma \ref{vanish} that the Chern numbers $\int_M c_1^3 c_2$ and $\int_M c_1^2 c_3$ vanish.
Then we use \cite[Theorem 1.2]{GS}, which computes the Chern number $\int_M c_1 c_4$ in terms of the number $N_i$ of fixed points with $i$ negative weights, and get a rational number for $\int_M c_1 c_2^2$. This contradicts the fact that each Chern number of a compact almost complex manifold is an integer. It follows that such a manifold $M$ does not exist.

\section*{Acknowledgements}

The author would like to thank Leonor Godinho and Silvia Sabatini for fruitful discussions and their warm hospitalities; this work was partially done during the author's visits to Instituto Superior T\'ecnico and University of Cologne. The author also thanks the anonymous referees for their valuable comments and suggestions, which helped improve the quality of the paper and led to a simpler proof of Theorem~\ref{main}.

\section{Background and preliminaries} \label{s2}

Let the circle act on a compact oriented manifold $M$. The \textbf{equivariant cohomology} of $M$ is
\begin{center}
$H_{S^1}^*(M) := H^*(M \times_{S^1} S^{\infty})$. 
\end{center}
The projection map $\pi:M \to \{\textrm{pt}\}$ induces a push-forward map
\begin{center}
$\displaystyle \int_M := \pi_* : H_{S^1}^i (M;\mathbb{Z}) \longrightarrow H^{i - \dim M} (\mathbb{CP}^\infty ; \mathbb{Z})$.
\end{center}
The Atiyah-Bott-Berline-Vergne localization theorem states that this push-forward map can be computed by the fixed point data.

\begin{theorem} [ABBV localization theorem]\cite{AB, BV} \label{t21} Let the circle act on a compact oriented manifold $M$. For any $\alpha \in H_{S^1}^*(M;\mathbb{Q})$,
\begin{center}
$\displaystyle \int_M \alpha = \sum_{F \subset M^{S^1}} \int_F \frac{\alpha|_F}{e_{S^1}(N_F)}$.
\end{center}
Here, the sum is taken over all fixed components, and $e_{S^1}(N_F)$ is the equivariant Euler class of the normal bundle of $F$.
\end{theorem}

For a compact almost complex manifold, the Hirzebruch $\chi_y$-genus is the genus associated to the power series $\frac{x(1+ye^{-x(1+y)})}{1-e^{-x(1+y)}}$. For a $2n$-dimensional compact almost complex manifold $M$, let $\chi_y(M) = \sum_{i=0}^n \chi^i(M) \cdot y^i$ denote the Hirzebruch $\chi_y$-genus of $M$. Here, $\chi^i(M)=\int_M T_i^n$, where $T_i^n$ is a rational combination of products of Chern classes $c_{j_1} \cdots c_{j_k}$ such that $j_1+\cdots+j_k=n$. Note that $\chi_{-1}(M)$ is the Euler characteristic of $M$, $\chi_0(M)=\textrm{Todd}(M)$ is the Todd genus of $M$, and $\chi_1(M)=\textrm{sign}(M)$ is the signature of $M$.

For a circle action on a compact almost complex manifold with a discrete fixed point set, the Atiyah-Singer index formula for the Hirzebruch $\chi_y$-genus is as follows.

\begin{theorem} \label{t22} \cite{HBJ, L} Let the circle act on a $2n$-dimensional compact almost complex manifold $M$ with a discrete fixed point set. Then for each $i$ such that $0 \leq i \leq n$,
\begin{center}
$\displaystyle \chi^i(M) = \sum_{p \in M^{S^1}} \frac{\sigma_i(g^{w_{p,1}},\cdots,g^{w_{p,n}})}{\prod_{j=1}^n (1-g^{w_{p,j}})}=(-1)^i N_i=(-1)^i N_{n-i}$.
\end{center}
Here, $\chi_y(M)=\sum_{i=0}^n \chi^i(M) \cdot y^i$ is the Hirzebruch $\chi_y$-genus of $M$, $g$ is an indeterminate, $\sigma_i$ is the $i$-th elementary symmetric polynomial in $n$ variables, and $N_i$ is the number of fixed points that have exactly $i$ negative weights.
\end{theorem}

Let the circle group act on a $2n$-dimensional compact almost complex manifold $M$. The equivariant Chern classes of $M$ are the ordinary Chern classes of the bundle $TM \times_{S^1} ES^1 \to M \times_{S^1} ES^1$. Let $j_1,j_2,\cdots,j_n$ be non-negative integers such that $j_1+2j_2+\cdots+nj_n=n$. We define the Chern number, also denoted $\int_M c_1^{j_1} c_2^{j_2} \cdots c_n^{j_n}$, by $\langle c_1^{j_1} c_2^{j_2} \cdots c_n^{j_n}, [M] \rangle$, where $[M]$ is the fundamental homology class of $M$.

Let $p$ be an isolated fixed point. The tangent space $T_pM$ at $p$ decomposes into complex 1-dimensional vector space $L_{p,1}$, $\cdots$, $L_{p,n}$, and the circle acts on each $L_{p,i}$ as multiplication by $g^{w_{p,i}}$ for all $g \in S^1 \subset \mathbb{C}$, where $w_{p,i}$ is a non-zero integer. These non-zero integers $w_{p,1}$, $\cdots$, $w_{p,n}$ are called the \textbf{weights} at $p$. The $i$-th equivariant Chern class $c_i$ at $p$ is
\begin{center}
$c_i|_p=\sigma_i(w_{p,1},\cdots,w_{p,n}) t^i$,
\end{center}
where $\sigma_i$ is the $i$-th elementary symmetric polynomial in $n$ variables, and $t$ is the degree 2 generator of $H^*(\mathbb{CP}^\infty;\mathbb{Z})$.

In \cite{GS}, Godinho and Sabatini computed the Chern number $\int_M c_1 c_{n-1}$ in terms of the number $N_i$ of fixed points with $i$ negative weights.

\begin{theorem} \cite{GS} \label{gsformula}
Let the circle act on a $2n$-dimensional compact almost complex manifold with a discrete fixed point set. For each $i$, let $N_i$ denote the number of fixed points with exactly $i$ negative weights. Then
\begin{center}
$\displaystyle \int_M c_1 c_{n-1}=\sum_{i=0}^n N_i \left[ 6i(i-1)+\frac{5n-3n^2}{2}\right].$
\end{center}
\end{theorem}

The Hirzebruch $\chi_y$-genus satisfies a property that two consecutive coefficients are non-zero.

\begin{lemma} \label{l27} \cite{JT, J2}
Let the circle act on a compact almost complex manifold with a non-empty discrete fixed point set such that $\dim M>0$. Then there exists $i$ such that $N_i \neq 0$ and $N_{i+1} \neq 0$, where $N_j$ is the number of fixed points that have exactly $j$ negative weights. Alternatively, there exists $i$ such that $\chi^i(M) \neq 0$ and $\chi^{i+1}(M) \neq 0$. \end{lemma}

For a manifold $M$, let $\chi(M)$ denote the Euler number of $M$. For a circle action on a compact manifold, the Euler number of the manifold is equal to the sum of Euler numbers of its fixed components.

\begin{theorem} \label{Euler} \cite{K}
Let the circle act on a compact manifold. Then
\begin{center}
$\displaystyle \chi(M)=\sum_{F \subset M^{S^1}} \chi(F)$.
\end{center}
\end{theorem}

For a circle action on a compact oriented manifold, if the action has an odd number of fixed points, then the dimension of the manifold is a multiple of 4.

\begin{lemma} \label{even}
Let the circle group act on a compact oriented manifold $M$ with a non-empty discrete fixed point set. If the dimension of $M$ is not divisible by 4, then the number of fixed points is even.
\end{lemma}

\begin{proof}
Since $M$ is compact and oriented and the dimension of $M$ is not a multiple of 4, the Euler number of $M$ is even. By Theorem \ref{Euler}, the Euler number of $M$ is equal to the sum of the Euler numbers of its fixed components, which are isolated fixed points. Because the Euler number of a point is 1, this lemma holds. 
\end{proof}

A classification of a circle action on a compact, connected almost complex manifold with few fixed points is as follows.

\begin{theorem} \label{few} \cite{J3}
Let the circle act on a compact, connected almost complex manifold $M$.
\begin{enumerate}
\item If the action has exactly one fixed point $p$, then $M$ is the point itself. That is, $M=\{p\}$.
\item If the action has exactly two fixed points, then $M$ is the 2-sphere, or $\dim M=6$ and 
\item If the action has exactly three fixed points, then $\dim M=4$
\end{enumerate}
\end{theorem}

\section{Chern numbers of high dimensional manifolds with 4 fixed points}

In this section, we use \cite[Lemma 3.2]{J5} and the ABBV localization theorem (Theorem \ref{t21}) to show Lemma \ref{vanish} that certain Chern numbers vanish for a circle action on a high dimensional compact almost complex manifold with 4 fixed points. First, we recall \cite[Lemma 3.2]{J5}.

\begin{lemma} \label{l32} \cite{J5}
Let $n \geq 4$. Let the circle act on a $2n$-dimensional compact almost complex manifold $M$ with 4 fixed points. Then one of the following holds:
\begin{enumerate}[(1)]
\item $\displaystyle \sum_{i=1}^n w_{p,i}=0$ for all $p \in M^{S^1}$ and $\displaystyle \sum_{p \in M^{S^1}} \frac{1}{\prod_{i=1}^n w_{p,i}}=0$.
\item We can divide the fixed points into 2 pairs $(q_1,q_2)$ and $(q_3,q_4)$ such that
\begin{enumerate}[(a)]
\item $\displaystyle  \prod_{i=1}^n w_{q_1,i}=-\prod_{i=1}^n w_{q_2,i}$,
\item $\displaystyle  \prod_{i=1}^n w_{q_3,i}=-\prod_{i=1}^n w_{q_4,i}$, and
\item $\displaystyle  \sum_{i=1}^n w_{q_1,i}=\sum_{i=1}^n w_{q_2,i}=-\sum_{i=1}^n w_{q_3,i}=-\sum_{i=1}^n w_{q_4,i}$.
\end{enumerate}
\end{enumerate}
\end{lemma}

In both cases of Lemma \ref{l32}, the following holds.

\begin{lemma} \label{divide}
Let $n \geq 4$. Let the circle act on a $2n$-dimensional compact almost complex manifold $M$ with 4 fixed points. Then we can divide the fixed points into 2 pairs $(q_1,q_2)$ and $(q_3,q_4)$ such that
\begin{center}
$\displaystyle  \sum_{i=1}^n w_{q_1,i}=\sum_{i=1}^n w_{q_2,i}=-\sum_{i=1}^n w_{q_3,i}=-\sum_{i=1}^n w_{q_4,i}$.
\end{center}
\end{lemma}

Using Lemma \ref{divide}, we show that certain Chern numbers vanish for a circle action on a high-dimensional compact almost complex manifold with 4 fixed points.

\begin{lemma} \label{pre-vanish}
Let $n \geq 4$. Let the circle act on a $2n$-dimensional compact almost complex manifold $M$ with 4 fixed points. Let $j_1,j_2,\cdots,j_n$ be non-negative integers such that $j_1+2j_2+\cdots+nj_n=n$. If $j_1 \geq 2$, then 
\begin{center}
$\displaystyle \int_M c_1^{j_1} c_2^{j_2} \cdots c_n^{j_n}=0$.
\end{center}
\end{lemma}

\begin{proof}
By Lemma \ref{divide}, we can divide the fixed points into 2 pairs $(q_1,q_2)$ and $(q_3,q_4)$ such that
\begin{center}
$\displaystyle  \sum_{i=1}^n w_{q_1,i}=\sum_{i=1}^n w_{q_2,i}=-\sum_{i=1}^n w_{q_3,i}=-\sum_{i=1}^n w_{q_4,i}$.
\end{center}
Let $a=\sum_{i=1}^n w_{q_1,i}$. Then $c_1|_{q_i}=(w_{q_i,1}+w_{q_i,2}+\cdots+w_{q_i,n})t=at$ for $i \in \{1,2\}$ and $c_1|_{q_i}=(w_{q_i,1}+w_{q_i,2}+\cdots+w_{q_i,n})t=-at$ for $i \in \{3,4\}$. In particular, $c_1^2|_{q_i}=a^2 t^2$ for all $i \in \{1,2,3,4\}$.

Next, by a dimensional reason that $\int_M$ is a map from $H_{S^1}^i (M;\mathbb{Z})$ to $H^{i - \dim M} (\mathbb{CP}^\infty ; \mathbb{Z})$, we get
\begin{center}
$\displaystyle \int_M c_1^{j_1-2} c_2^{j_2} \cdots c_n^{j_n}=0$.
\end{center}
On the other hand, taking $\alpha=c_1^{j_1-2} c_2^{j_2} \cdots c_n^{j_n}$ in Theorem \ref{t21},
\begin{center}
$\displaystyle \int_M c_1^{j_1-2} c_2^{j_2} \cdots c_n^{j_n}=\sum_{i=1}^4 \frac{(c_1^{j_1-2} c_2^{j_2} \cdots c_n^{j_n})|_{q_i}}{e_{S^1}(N_{q_i})}$.
\end{center}
Finally, taking $\alpha=c_1^{j_1} c_2^{j_2} \cdots c_n^{j_n}$ in Theorem \ref{t21},
\begin{center}
$\displaystyle \int_M c_1^{j_1} c_2^{j_2} \cdots c_n^{j_n}=\sum_{i=1}^4 \frac{(c_1^{j_1} c_2^{j_2} \cdots c_n^{j_n})|_{q_i}}{e_{S^1}(N_{q_i})}=\sum_{i=1}^4 c_1^2|_{q_i} \frac{(c_1^{j_1-2} c_2^{j_2} \cdots c_n^{j_n})|_{q_i}}{e_{S^1}(N_{q_i})}=\sum_{i=1}^4 a^2 t^2 \frac{(c_1^{j_1-2} c_2^{j_2} \cdots c_n^{j_n})|_{q_i}}{e_{S^1}(N_{q_i})}=a^2 t^2 \sum_{i=1}^4  \frac{(c_1^{j_1-2} c_2^{j_2} \cdots c_n^{j_n})|_{q_i}}{e_{S^1}(N_{q_i})}=a^2 t^2 \int_M c_1^{j_1-2} c_2^{j_2} \cdots c_n^{j_n}=0$.
\end{center}
Thus, this lemma holds.
\end{proof}

In particular, we obtain the following.

\begin{lemma} \label{vanish}
Let the circle act on a $10$-dimensional compact almost complex manifold $M$ with 4 fixed points. Then
\begin{center}
$\displaystyle \int_M c_1^3 c_2=0$ and $\displaystyle \int_M c_1^2 c_3=0$.
\end{center}
\end{lemma}

\begin{proof}
This follows from Lemma \ref{pre-vanish}.
\end{proof}

\section{Proof of Theorem \ref{main}}

In this section, we prove Theorems~\ref{bound} and Theorem~\ref{main}. First, we prove Theorem~\ref{main}.

\begin{proof}[\textbf{Proof of Theorem \ref{main}}]
Assume on the contrary that there exists a circle action on a 10-dimensional compact almost complex manifold $M$ with exactly 4 fixed points. The Todd genus $\mathrm{Todd}(M)$ of $M$ is
\begin{center}
$\displaystyle \mathrm{Todd}(M)=\int_M \frac{1}{1440} (-c_1c_4+c_1^2c_3+3c_1c_2^2-c_1^3c_2)$,
\end{center}
see \cite[p. 14]{Hi}.
By Lemma \ref{vanish}, $\int_M c_1^2 c_3=0$ and $\int_M c_1^3 c_2=0$. Thus,
\begin{center}
$\displaystyle \mathrm{Todd}(M)=\int_M \frac{1}{1440} (-c_1c_4+3c_1c_2^2).$
\end{center}
Moreover, since the Todd genus $\mathrm{Todd}(M)$ of $M$ is the Hirzebruch $\chi_y$-genus $\chi_y(M)$ of $M$ evaluated at $y=0$, by Theorem \ref{t22}, $\mathrm{Todd}(M)=N_0$.
Therefore, the following equation holds.
\begin{center}
$\displaystyle \int_M 3c_1 c_2^2=1440 \cdot N_0 + \int_M c_1c_4$.
\end{center}
Since any Chern number of a compact almost complex manifold is an integer,
$$\int_M 3c_1 c_2^2 \equiv 0 \pmod 3.$$
By Theorem~\ref{gsformula}, $\int_M c_1c_4 = \sum_{i=0}^5 N_i [ 6i(i-1)-25]$ and hence
\begin{center}
$\D \int_M 3c_1 c_2^2=1440 \cdot N_0 + \int_M c_1c_4 = 1440 \cdot N_0 + \sum_{i=0}^5 N_i \left[ 6i(i-1)-25\right] \equiv - \sum_{i=0}^5 N_i =-4 \pmod 3,$
\end{center}
which is a contradiction. Therefore, this theorem holds.
\end{proof}

\begin{proof}[Proof of Theorem~\ref{bound}]
This theorem holds if there is a fixed component of positive dimension. Thus, from now on, suppose that the fixed point set is discrete. By Lemma \ref{even}, the number of fixed points is even. By Theorem \ref{few}, the action cannot have exactly two fixed points. By Theorem \ref{main}, the action cannot have exactly four fixed points. Therefore, this theorem holds.
\end{proof}

\end{document}